\documentclass{amsart}

\usepackage{amsmath,amssymb,amsthm,enumerate}
\usepackage[pagebackref]{hyperref}
\usepackage[initials,backrefs]{amsrefs}

\newtheorem{thm}{Theorem}[section]

\newtheorem{prop}[thm]{Proposition}

\newtheorem*{HenselLemma}{Hensel's Lemma}

\theoremstyle{remark}

\newtheorem{example}[thm]{Example}

\theoremstyle{definition}
\newtheorem{defn}[thm]{Definition}

\newcommand{\set}[2]{\left\{ #1 \mid #2\right\}}
\newcommand{\ex}[2]{\exists #1 \left( #2 \right)}
\newcommand{\all}[2]{\forall #1 \left( #2 \right)}

\newcommand{\Bga}{B_{\geq\gamma}(\alpha)}

\newcommand{\cO}{\mathcal{O}}
\newcommand{\maxm}{\mathfrak{m}}
\newcommand{\rvd}{\rv_\delta}
\newcommand{\rvg}{\rv_\gamma}
\newcommand{\RVd}{\RV_\delta}
\newcommand{\RVg}{\RV_\gamma}

\newcommand{\opd}{\oplus_\delta}
\newcommand{\opg}{\oplus_\gamma}
\newcommand{\ep}{\varepsilon}

\newcommand{\bbN}{\mathbb{N}}
\newcommand{\bbZ}{\mathbb{Z}}
\newcommand{\bbQ}{\mathbb{Q}}

\newcommand{\bfx}{\mathbf{x}}
\newcommand{\bfy}{\mathbf{y}}
\newcommand{\bfz}{\mathbf{z}}

\newcommand{\bfu}{\mathbf{u}}
\newcommand{\bfw}{\mathbf{w}}

\DeclareMathOperator{\chr}{char}
\DeclareMathOperator{\res}{res}
\DeclareMathOperator{\RV}{RV}
\DeclareMathOperator{\rv}{rv}
\DeclareMathOperator{\acl}{acl}

\begin{document}

\title{Relative decidability and definability in henselian valued fields}

\author{Joseph Flenner}
\address{University of Notre Dame \\
Department of Mathematics \\ 
255 Hurley Hall \\ 
Notre Dame, IN 46556 \\
U.S.A.} 

\email{jflenner@nd.edu}

\date{\today}

\begin{abstract}
Let $K$ be a henselian valued field of characteristic $0$.  Then $K$ admits a definable partition on each piece of which the leading term of a polynomial in one variable can be computed as a definable function of the leading term of a linear map.  Two applications are given: first, a constructive quantifier elimination relative to the leading terms, suggesting a relative decision procedure; second, a presentation of every definable subset of $K$ as the pullback of a definable set in the leading terms subjected to a linear translation.
\end{abstract}

\maketitle

\section{Introduction}

In \cite{hol1}, Holly showed that definable subsets of algebraically closed valued fields can be expressed canonically as disjoint unions of \emph{swiss cheeses}, sets of the form
$$S \setminus (T_1 \cup \ldots \cup T_n)$$
where the $S, T_i$ are open or closed balls.  In this way she presents the balls as the basic building blocks of the definable subsets of the field $K$.  The language of valued fields used here is a three-sorted one, with sorts for the field, the value group, and the residue field.

Holly's theorem relied essentially on the completeness and quantifier elimination in the theory of algebraically closed valued fields (ACVF) dating from Robinson \cite{rob1}.  As made explicit in \cite{hol2}, this was intended as a first step towards the elimination of imaginaries for ACVF that came to fruition in work of Haskell, Hrushovski, and Macpherson in \cite{hhm1}.  This in turn became a starting point for a line of work establishing ACVF as a testing ground for the adaptation of methods from stability theory to nonstable theories.  See for example the monograph \cite{hhm2} of Haskell, Hrushovski, and Macpherson.

Meanwhile, model-theoretic work on the $p$-adics has paralleled to some degree work on ACVF.  We have, for example, the decision procedure of Cohen \cite{coh1} and quantifier elimination of Macintyre \cite{mac1}.  Macintyre's theorem exists in a language enhancing the usual valued field language by a system of predicates identifying the $n^\text{th}$ powers for each $n$.  Translated into the value group $(\bbZ, +)$, this evokes the divisibility predicates which are precisely what is needed to achieve quantifier elimination in Presburger arithmetic.

Indeed, it appears that in the general setting of henselian valued fields, many of the sort of results holding outright in ACVF can be proved, in a sense, \emph{modulo} the associated theories of residue field and value group.  This idea began with Ax-Kochen \cite{ak1,ak2,ak3} and Ersov, establishing the completeness of the theory of henselian valued fields of pure characteristic $0$ relative to the theories of the residue field and value group.

For general henselian fields, however, quantifier elimination relative to the residue field and value group fails.  Suggesting that perhaps the three-sorted language employed by Holly is not optimal in the henselian case, Kuhlmann \cite{kuh1} has obtained elimination of quantifiers relative to an associated structure of `additive and multiplicative congruences'.

We aim to prove for henselian valued fields of characteristic $0$ analogues of both Holly's theorem on canonical forms of subsets of the field and of Cohen's on decidability.  To do so we adopt a language built around structures of leading terms which is equivalent to (but for our purposes somewhat more syntactically convenient than) Kuhlmann's.  These both capture the information of the value group and residue field, and provide an algebraic view of the topology of balls.

Section \ref{leadingterms} covers the relevant definitions and basic properties.  The main technical tool is in Section \ref{decomposition}, in which it is shown that the field admits a partition on each piece of which the leading term of a polynomial in $x$ can be easily computed in terms of the leading term of $x-a$, some $a\in K$.  Section \ref{qe} uses this to describe a constructive relative quantifier elimination procedure (differing in particular from Kuhlmann's result in its constructivity), while Section \ref{subsets} concludes with a characterization of the definable subsets (in one variable) of the field relative to the definable subsets of the leading term structures.

While elimination of imaginaries has already been generalized from ACVF to the $p$-adics \cite{hm1} (as well as real-closed valued fields \cite{mel1}), it is hoped that this may eventually form the one-dimensional case for a more native and comprehensive approach to a relative elimination of imaginaries for henselian valued fields in characteristic $0$.

\subsection{Acknowledgments}  The bulk of the research presented here was done while the author was a graduate student at the University of California, Berkeley under the supervision of Thomas Scanlon.  My gratitude for Dr. Scanlon's advice and insight through countless discussions on this subject (and many others as well) cannot be overstated.  I would also like to thank Deirdre Haskell, Dugald Macpherson, and Anand Pillay for helpful conversations and support during and after visits to Hamilton and Leeds.  Some loose ends were tied up and a first draft written while I was hosted by the Hausdorff Research Institute for Mathematics, whose hospitality during their trimester program on Diophantine Equations I happily acknowledge.

\section{Leading terms}
\label{leadingterms}

\subsection{Definitions and notation}
\label{defandnot}

To fix notation, we work in a valued field $K$ with 
value group $V$ 
and valuation ring $\cO:=\set{x\in K}{v(x)\geq 0}$.
Among the ideals of $\cO$ are 
$$\maxm_\delta := \set{x\in\cO}{v(x)>\delta}$$ 
and in particular the (unique) maximal ideal $\maxm := \maxm_0$.

The residue field is $R:=\cO/\maxm$, and the residue of $x$ is written either 
$\bar{x}$ or $\res(x)$ as convenient.  More generally, for any $\delta\geq 0$ in $V$ we have the ring $R_\delta:=\cO/\maxm_\delta$ with reduction map $\res_\delta:\cO\rightarrow R_\delta$.

Valued fields possess a topology having as basic open sets the open balls 
$$B_{> \delta}(a) := \set{x \in K}{v(x-a) > \delta}$$ 
with center $a$ and radius $\delta$.  
Closed balls $B_{\geq\delta}(a)$ are defined in the obvious way, and we will also have occasion to refer to balls of the form
$$B_{>\delta / n}(a) := \set{x \in K}{n v(x-a) > \delta}$$
even if $\delta$ is not divisible by $n$ in $V$.  We allow the radius $\delta$ to be either $\infty$ or $-\infty$, so $K$, $\emptyset$, and $\left\{a\right\}$ are all balls.

It is readily shown using the ultrametric inequality $v(x+y)\geq\min\left\{v(x),v(y)\right\}$ that for any two balls $B$ and $C$, if $B\cap C\neq\emptyset$ then $B\subseteq C$ or $C\subseteq B$; that \emph{any} element of $B$ is a center of $B$; and that both the open and closed balls are in fact clopen in the valuation topology.

\begin{defn}
\label{defleadterm}
Let $\delta \geq 0$ in $V$.  The \emph{leading term structure of order $\delta$} is the quotient group
$$\RVd := K^\times / (1+\maxm_\delta).$$
The quotient map is denoted $\rvd : K^\times \rightarrow \RVd$.  
As with the value group, it is convenient to include an element $\infty$ in $\RVd$ as $\rvd(0)$.  Generally, the subscript $0$ will be omitted, so $\RV=\RV_0$ and $\rv=\rv_0$.

Besides the induced multiplication, $\RVd$ inherits a partially defined addition from $K$ via the relation
$$\oplus_\delta(\mathbf{x},\mathbf{y},\mathbf{z}) \Longleftrightarrow
\ex{x,y,z \in K}{\mathbf{x}=\rvd(x) \wedge \mathbf{y}=\rvd(y) \wedge \mathbf{z}=\rvd(z) \wedge x+y=z}.$$ 

The sum $\bfx + \bfy$ is said to be \emph{well-defined} (and $=\bfz$) if there is exactly one $\bfz$ such that $\opd(\bfx,\bfy,\bfz)$.
While the notation $\bfx+\bfy=\bfz$ will be used exclusively when well-defined, in order to better accommodate sums of more than two terms it will be useful to write $\bfx + \bfy\approx\bfz$ for $\opd(\bfx,\bfy,\bfz)$ in general, bearing in mind that $\bfx + \bfy\approx\bfz$ and $\bfx + \bfy\approx\mathbf{w}$ does not imply $\mathbf{z}=\mathbf{w}$.

If $\gamma \geq \delta \geq 0$, since $1+\maxm_\gamma \subseteq 1+\maxm_\delta$ there is a natural map $\RVg\rightarrow\RVd$, which we also denote $\rvd$, or $\rv_{\gamma\rightarrow\delta}$ should there be fear of confusion.

To be clear, then, the \emph{leading term language} refers to a multisorted language
$$\left(K,\langle\RV_{\delta}\rangle_{\delta\in\Delta}\right)$$
with the usual ring language on the field sort, $\Delta\subseteq\set{\delta\in V}{0\leq\delta <\infty}$ to be specified as needed, the multiplication and the relation $\opd$ on each $\RVd$, and as maps between the sorts $\rv_\delta: K\rightarrow\RVd$ and $\rv_{\gamma\rightarrow\delta}:\RVg\rightarrow\RVd$ for each $\gamma\geq\delta\in\Delta$.
\end{defn}

The following propositions justify some of the claims of the Introduction.
The proofs follow directly from the definitions.

\begin{prop}
\label{prop1}
Given $0\leq\delta\in V$, the following are equivalent for all nonzero $x,y\in K$: 
\begin{enumerate}
\item
$\rvd(x)=\rvd(y)$
\item
$v(x-y)>v(y)+\delta$
\item
$\res_\delta(x/y)=1$ in $R_\delta$
\item
$B_{>v(x)+\delta}(x)=B_{>v(y)+\delta}(y)$
\qed
\end{enumerate}
\end{prop}

In particular, note that because $v(x-y)>v(y)$ can occur only when $v(x)=v(y)$, $\rvd(x)=\rvd(y)$ implies $v(x)=v(y)$.  Thus we can speak unambiguously of $v(\mathbf{x})$ for $\mathbf{x}\in\RVd$ (any $\delta \geq 0$).

The following example provides a good general source of intuition.

\begin{example}
\label{exam1}
Let $R$ be any field, and $V$ any ordered abelian group.  The \emph{Hahn field} $R((t^V))$ consists of the formal power series over $R$
$$\sum_{\delta\in V} c_\delta t^\delta$$
where the support $\set{\delta}{c_\delta \neq 0}$ is well-ordered.  Taking 
$v\left(\sum c_\delta t^\delta\right)=\min\left\{\delta\mid c_\delta \neq 0\right\}$, $R((t^V))$ has residue field $R$ and value group $V$.

More concretely, in case $R=\bbQ$ and $V=\bbZ$, we have the field $\bbQ((t))$ of Laurent series over the rational numbers.  Two such series will have the same leading term of order $3$, say, if they have the same value and their first four coefficients coincide.  Thus, if
$$\begin{array}{rcl}
x & = & t^{-2}+t^{-1}+1+t+2t^2+t^3+\ldots \\
y & = & t^{-2}+t^{-1}+1+t+t^2+t^3+\ldots
\end{array}$$
then $\rv_3(x) = \rv_3(y)$
since $v(x)=v(y)=-2$ and $v(x-y)=v(t^2)=2>v(y)+3$.  But $\rv_4(x)\neq\rv_4(y)$.
\end{example}

Next we establish when the addition on $\RVd$ is well-defined.

\begin{prop}
\label{prop3}
Let $\delta\geq 0$, and $v(x+y)=\min\{v(x),v(y)\}$.  Then for all $z$ such that $\rvd(z)=\rvd(x)$, $\rvd(z+y)=\rvd(x+y)$.

Conversely, if $v(x+y)>v(x)$, then there exists $z$ such that $\rvd(z)=\rvd(x)$ but $\rvd(z+y)\neq\rvd(x+y)$.
\end{prop}

\begin{proof}
Consider $z=x(1+m)$, with $v(m)>\delta$.  Defining $m':=\frac{xm}{x+y}$, we then find
$$z+y=x(1+m)+y=x+y+(x+y)m'=(x+y)(1+m')$$
and
$$v(m')=v(m)+v(x)-v(x+y)\geq v(m)>\delta.$$

On the other hand, suppose $v(x+y)-v(x)=\ep>0$, and let $m$ be any element of value $\delta+\ep$.  Take $z:=x(1+m)$.  As $v(m)>\delta$, $\rvd(z)=\rvd(x)$.  
But 
$$v((z+y)-(x+y))=v(z-x)=v(x)+v(m)=v(x+y)+\delta$$ 
implies, by Proposition \ref{prop1}, that $\rvd(z+y)\neq\rvd(x+y)$.
\end{proof}

Therefore, there is a well-defined $\mathbf{z} \in \RVd$ such that
$\opd(\rvd(x),\rvd(y),\mathbf{z})$ precisely when $v(x+y)=\min\{v(x),v(y)\}$,
namely $\mathbf{z}=\rvd(x+y)$.

For later use, it will be necessary to extend \ref{prop3} to sums of more than two terms in $\RVd$.  This is not entirely automatic, since even if say $v(x+y+z)=\min\left\{v(x),v(y),v(z)\right\}$, it may be the case that $\rvd(y)+\rvd(z)$ is not well-defined.  It must then be shown that if $\opd(\rvd(y),\rvd(z),\mathbf{u}_1)$ and $\opd(\rvd(y),\rvd(z),\mathbf{u}_2)$ with $\mathbf{u}_1\neq\mathbf{u}_2$, we still have $\rvd(x)+\mathbf{u}_1=\rvd(x)+\mathbf{u}_2$.  This however is easily accomplished with help from Proposition \ref{prop1}.

\begin{prop}
\label{prop3.5}
Suppose that $v(x_1+\ldots+x_n)=\min\left\{v(x_1),\ldots,v(x_n)\right\}$.  Then $\bfy\approx\rvd(x_1)+\ldots+\rvd(x_n)$ if and only if $\mathbf{y}=\rvd(x_1+\ldots+x_n)$.
\qed
\end{prop}

The next proposition clarifies what happens when the addition is not well-defined.

\begin{prop}
\label{prop4}
Suppose that $v(x_1+\ldots+x_n)-\min\left\{v(x_i)\right\}=\ep > 0$.
If $\gamma \geq \delta + \ep$ and
$\rvg(x_1)+\ldots+\rvg(x_n)\approx\bfz \in\RVg$,
then $\rv_{\gamma\rightarrow\delta}(\bfz)=\rvd(x_1+\ldots+x_n)$.
\end{prop}

\begin{proof}
By definition of $\opg$,
there are $z\in K$ and $m_i\in\maxm_\gamma$
such that $\bfz=\rvg(z)$ and $z=x_1(1+m_1)+\ldots+x_n(1+m_n)$.
Now 
$$v(x_1+\ldots+x_n-z)=v(x_1m_1+\ldots+x_nm_n)\geq \min\left\{v(x_im_i)\right\}$$ 
$$>\min\left\{v(x_i)\right\}+\gamma \geq \min\left\{v(x_i)\right\}+\ep+\delta=v(x_1+\ldots+x_n)+\delta$$
and Proposition \ref{prop1} give $\rvd(x_1+\ldots+x_n)=\rvd(z)=\rv_{\gamma\rightarrow\delta}(\bfz)$.
\end{proof}

In other words, when $v(x+y)>v(x)$, while \ref{prop3} shows that there is more
than one $\mathbf{z}\in\RVg$ such that $\rvg(x)+\rvg(y)\approx\mathbf{z}$, 
\ref{prop4} implies that all such $\mathbf{z}$ have the same image in $\RVd$ for $\delta\leq \gamma-(v(x+y)-v(x))$.

As a corollary, the following proposition shows that when $v(x+y)$ is not too much larger than $v(x)$ (compared to $\gamma$), at least $v(\rvg(x)+\rvg(y))$ is well-defined.  On the other hand, when $v(x+y)>v(x)+\gamma$, nothing further can be said.

\begin{prop}
\label{prop4.5}
Suppose $\ep=v(x+y)-v(x)\geq 0$.  Then
\begin{enumerate}[(i)]
\item
if $\gamma \geq \ep$ and $\opg(\rvg(x),\rvg(y),\bfz_1)$ and $\opg(\rvg(x),\rvg(y),\bfz_2)$, then $v(\bfz_1)=v(\bfz_2)$.
\item
if $0\leq\gamma < \ep$ and $v(z)>v(x)+\gamma$, then 
$\opg(\rvg(x),\rvg(y),\rvg(z))$.
\end{enumerate}
\end{prop}

\begin{proof}
The first statement is an immediate consequence of \ref{prop4} with $\delta = 0$, while the second follows from $\rvg(x)=\rvg(x+z)$, $\rvg(y)=\rvg(-x)$.
\end{proof}

\subsection{Interpretations}

Recall that a structure $N$ is \emph{interpretable} in $M$ over $A\subseteq M$ when there is an $A$-definable subset $S \subseteq M^n$
and an $A$-definable equivalence relation $\sim$ on $S$ such that
\begin{enumerate}[(i)]
\item
the elements of $N$ are in bijection with the equivalence classes of $\sim$, and
\item
the relations on $S$ induced by the relations and functions of $N$ by this bijection are all $A$-definable.
\end{enumerate}

As suggested by Proposition \ref{prop1}, the leading term structures in a sense encompass both residue field and value group.  This can now be made more explicit.

\begin{prop}
\label{prop5}
Let $0\leq\delta\in V$ and $\mathbf{d}\in\RVd$ be any element with $v(\mathbf{d})=\delta$.
\begin{enumerate}
\item
The value group $V$ is interpretable in $\RVd$ over $\{\mathbf{d}\}$.
\item
The ring $R_\delta$ is interpretable in $\RVd$ over $\{\mathbf{d}\}$.
\item
For $\gamma>\delta$, $\RVd$ is interpretable in $\RVg$ over $\{\mathbf{d}\}$.
\end{enumerate}
\end{prop}

\begin{proof}
(1): To begin, observe that $v(\bfx)>0$ is definable in $\RVd$.  Indeed, it is easily verified that
$$v(\bfx)>0 \Longleftrightarrow \mathbf{d}\bfx+\mathbf{1}=\mathbf{1}$$
(where $\mathbf{1}=\rvd(1)$).
From this it follows that $v(\bfx)=0$ is also definable:
$$v(\bfx)=0 \Longleftrightarrow \neg v(\bfx)>0 \wedge 
\ex{\bfy}{\bfx\bfy=\mathbf{1} \wedge \neg v(\bfy)>0}.$$

Now define the equivalence relation $\sim$ on $\RVd$ by
$$\bfx\sim\bfy \Longleftrightarrow 
\ex{\bfu}{v(\bfu)=0 \wedge \bfx=\bfu\bfy}.$$
Clearly, we have $\bfx\sim\bfy$ iff $v(\bfx)=v(\bfy)$, so that the equivalence classes of $\sim$ in $\RVd$ are in bijection with $V$.

Moreover, addition of $v(\bfx)+v(\bfy)$ in $V$ corresponds to the multiplication 
$\bfx\bfy$ in $\RVd$, 
and the group ordering $<$ is defined by $\bfx<\bfy$ iff 
$\bfx\neq\infty\wedge \bfx+\mathbf{d}\bfy=\bfx$.

(2): Define $\sim$ on $\RVd^+:=\set{\bfx\in\RVd}{v(\bfx)\geq 0}$ by 
$$\bfx\sim\bfy \Longleftrightarrow \ex{\bfz}{v(\bfz)>\delta\wedge \bfx-\bfy=\bfz}.$$
We leave it to the reader to confirm that elements of $R_\delta$ are in bijection with the $\sim$-equivalence classes in $\RVd^+$, with $\res_\delta(x)$ corresponding to $\rvd(x)/\sim$.
The multiplication and addition in $R_\delta$ translates directly from multiplication and addition in $\RVd$.

(3):  Considering $\mathbf{x}=\rvg(x),\mathbf{y}=\rvg(y)\in\RVg$, it will be enough to show that 
$\mathbf{x}\sim\mathbf{y} \Leftrightarrow \rvd(x)=\rvd(y)$ is definable over $\{\mathbf{d}\}$ in $\RVg$.  
Recalling \ref{prop4.5}, this follows from
$$\begin{array}{rcl}
\rvd(x)=\rvd(y) & \Longleftrightarrow & v(x-y)>v(y)+\delta \\ 
 & \Longleftrightarrow & \ex{\bfz\in\RVg}{v(\bfz)>v(\bfy)+\delta\wedge\bfx-\bfy\approx\bfz}.
 \end{array}$$
\end{proof}

In the following sections, we will only need to consider $\RVd$ when $\delta$ is the value of an integer.  Then the $\mathbf{d}$ in \ref{prop5} would always be $\emptyset$-definable, in which case the interpretations could in fact be taken over $\emptyset$.

As noted in the Introduction, in \cite{kuh1} Kuhlmann has introduced the `structures of additive and multiplicative congruences' which connect the structure of $R_\delta$ and $V$ in a similar way.
He defines, for each $\delta\geq 0$ in $V$, the system
$$K_\delta := \left(R_\delta, \RVd, \Theta_\delta(x,y)\right)$$
(whereby $\RVd$ is taken only as a multiplicative group).  
The relation $\Theta_\delta$ is defined on $R_\delta\times \RVd$ as
$$\Theta_\delta(x,y) \Leftrightarrow \exists z\in\cO\ \left(\res_\delta(z)=x \wedge \rvd(z)=y\right).$$

It can also be shown that $\RVd$ interprets $K_\delta$ and vice versa.  In fact, the two structures bear the stronger mutual relation of (quantifier-free) \emph{bi-interpretability}.  Though we find the formalism of the leading term language more convenient, the two languages should be taken as equivalent.  Cluckers and Loeser \cite{cl1} and Hrushovski and Kazhdan \cite{hk1} each work with other alternative manifestations of the leading term structures.

\subsection{Henselian fields}

The valued field $K$ is called \emph{henselian} if it satisfies

\begin{HenselLemma}
For all $P(x)\in\cO[x]$ and $a\in\cO$, if $v(P(a))>0$ and $v(P'(a))=0$, then there exists $b\in\cO$ such that $P(b)=0$ and $\bar{a}=\bar{b}$.
\end{HenselLemma}

For examples of henselian fields, in addition to the $p$-adics $\bbQ_p$ we have the Hahn fields $R((t^V))$ of Example \ref{exam1}.  This shows in particular that from an arbitrary field $R$ and ordered abelian group $V$, a henselian field can be constructed with $R$ and $V$ as residue field and value group.

It is well known that Hensel's Lemma can be reformulated to loosen the restriction on $v(P'(a))$ as in the following proposition.  See \cite{rib2} for a proof, as well as a thorough exposition of other equivalent forms of Hensel's Lemma.

\begin{prop}
\label{prop6}
Suppose $K$ is henselian, $P(x)\in\cO[x]$ and $a\in\cO$.  If $v(P(a))>2v(P'(a))$, then there exists $b\in\cO$ such that $P(b)=0$ and $\bar{a}=\bar{b}$.
\qed
\end{prop}

Both of these guarantee the existence of a root $b$ of $P$ close to the `approximate root' $a$, in the sense that $v(a-b)>0$.  In working with the leading term structures, it will be desirable to refine the conclusion that $a$ and $b$ have the same residue to give $\rvd(a)=\rvd(b)$ (note that $\bar{a}=\bar{b}$ implies $\rv(a)=\rv(b)$ only when $v(a)=0$).  A sharper result on the proximity of the approximate root to an actual root is obtained in

\begin{prop}
\label{prop7}
Suppose $K$ is henselian, $P(x)\in\cO[x]$, $a\in\cO$, and $0\leq\delta\in V$.  If $v(P(a))>2v(P'(a))+\delta$, then there exists $b\in\cO$ such that $P(b)=0$ and $v(a-b)>\delta$.
\end{prop}

\begin{proof}
By induction on $d:=\deg(P)$.  Let $b$ be the root of $P$ given by Proposition \ref{prop6}, and factor $P(x)=(x-b)R(x)$.  We have
$$v(P(a))=v(a-b)+v(R(a)),$$
$$v(P'(a))=v((a-b)R'(a)+R(a))\geq\min\left\{v(a-b)+v(R'(a)),v(R(a))\right\}.$$

Assume first that $v(R(a))\leq v(a-b)+v(R'(a))$.  Then $v(P(a))>2v(P'(a))+\delta$ gives
$$v(a-b)+v(R(a))>2v(R(a))+\delta,$$
whence $v(a-b)>v(R(a))+\delta\geq\delta$.

If on the other hand $v(a-b)+v(R'(a))<v(R(a))$, $v(P'(a))=v(a-b)+v(R'(a))$ implies
$$v(R(a))>2v(R'(a))+v(a-b)+\delta>2v(R'(a))+\delta.$$
Now the induction gives a root $c$ of $R$, and so also of $P$, such that $v(a-c)>\delta$.
\end{proof}

Therefore, to produce a root $b$ with $\rvd(a)=\rvd(b)$, it would suffice to require that $v(P(a))>2v(P'(a))+\gamma$ with $\gamma$ at least $v(a)+\delta$.

\section{Decomposition}
\label{decomposition}

\subsection{Collisions}

From now on, the valued field $K$ is assumed to be henselian and of characteristic $0$.  The residue field may have positive characteristic, though the results generally take a simpler form in the pure characteristic $0$ case.

The goal being to investigate definability in $K$ through the leading term structures,
this would be trivial if we could simply say for $f(x)\in K[x]$ that $\rv(f(x))=f(\rv(x))$.
However, as seen in Proposition \ref{prop3}, this is not always the case.  For example, 
$\rv(x^2+a)$ is identically equal to $\rv(x)^2+\rv(a)$ only when the sum is well-defined.  Difficulties arise wherever $x^2$ and $a$ `collide' to make $v(x^2+a)>\min\{v(x^2),v(a)\}$.

Our strategy is to partition $K$ so that on each piece of the partition, 
$v(f(x))$ reduces to a simple form and $\rvd(f(x))$ can be analyzed within $\RVd$ in linear terms as a function of $\rvd(x-\alpha)$ for some $\alpha\in K$.

\begin{defn}
\label{defcollision}
Say $f(x)=\sum\limits_{i=0}^da_i(x-\alpha)^i$ 
has a \emph{collision at $\beta$ around $\alpha$} if 
$v(f(\beta))>\min\limits_{i\leq d}\{ v(a_i(\beta-\alpha)^i)\}$.
In this case, the \emph{severity} of the collision is the value
$$v(f(\beta))-\min\limits_{i\leq d}\{ v(a_i(\beta-\alpha)^i)\}.$$
\end{defn}

Note that it is impossible for a polynomial to have a collision at $\alpha$ around $\alpha$,even if $\alpha$ is a root of $f(x)$.  On the other hand, if $\alpha\neq\beta$, then $f(\beta)=0$ iff $f(x)$ has a collision of \emph{infinite} severity at $\beta$ around $\alpha$.

As mentioned above, by Proposition \ref{prop3.5} for any $\beta$ where $f(x)$ does \emph{not} have a collision,
$$\rvd(f(\beta))=\sum\limits_{i=0}^d\rvd(a_i)\rvd(\beta-\alpha)^i$$ 
is well-defined.
Accordingly, the existence of a collision at $\beta$ around $\alpha$ depends only on $\rv(\beta-\alpha)$.

In fact, we can go further by locating collisions near roots of the derivatives of $f(x)$.  Here let us introduce the convention that if $\deg(f)=d$ then by the \emph{derivatives of $f(x)$} we mean $f$, $f'$, \ldots, and $f^{(d)}$, notably including $f$ itself as the `$0^{\text{th}}$ derivative'.

By scaling $f(x)\in K[x]$ to obtain a polynomial $P(x)$ over $\cO$, it is possible to transfer Hensel's Lemma to polynomials over the field rather than only the valuation ring, with collisions filling the role of the conditions on the valuation of $P$.  The following can be seen as a further generalization of the Hensel property along these lines, giving a root of a derivative of $f$ wherever $f$ has a collision exceeding a bound on the severity.

\begin{prop}
\label{prop21}
Let $\alpha\in K$ and $f(x)=\sum\limits_{i=0}^da_i(x-\alpha)^i$.  Suppose moreover that $f$ has a collision at $\beta$ around $\alpha$ of severity $\ep > 2^m(v(m!)+\delta)$, where 
$$m=\max\set{i\leq d}{\all{j\leq d}{v(a_i(\beta-\alpha)^i)\leq v(a_j(\beta-\alpha)^j)}}.$$
Then there is a $\lambda\in K$ and $n<m$ such that $f^{(n)}(\lambda)=0$ and $\rvd(\lambda-\alpha)=\rvd(\beta-\alpha)$.
\end{prop}

\begin{proof}
Note first that $\beta\neq\alpha$, since otherwise $v(f(\beta))=v(a_0)$ and $f$ cannot have a collision at $\beta$.  

Define $\sigma:=a_m(\beta-\alpha)^m$ and
$$P(x):=\frac{f((\beta-\alpha)x+\alpha)}{\sigma}=\frac{1}{\sigma}\sum\limits_{i=0}^d a_i(\beta-\alpha)^ix^i.$$
So, $P(x)\in\cO[x]$ and $v(P(1))=\ep$.

Consider $P^{(m)}(1)$.  Since
$$P^{(m)}(1)=\frac{1}{\sigma}\sum_{i=m}^d\frac{i!}{(i-m)!}a_i(\beta-\alpha)^i$$
for $i=m$ we have
\begin{equation*}
v\left(\frac{1}{\sigma}\frac{i!}{(i-m)!}a_i(\beta-\alpha)^i\right)=v\left(\frac{m!}{\sigma}a_m(\beta-\alpha)^m\right)=v(m!)
\end{equation*}
while for $i>m$,
\begin{equation}
\label{eqn22}
v\left(\frac{1}{\sigma}\frac{i!}{(i-m)!}a_i(\beta-\alpha)^i\right)=
v\left(\frac{i!}{(i-m)!}\right)+v\left(a_i(\beta-\alpha)^i\right)-v\left(a_m(\beta-\alpha)^m\right).
\end{equation}
Since $m!$ divides $i!/(i-m)!$, and $v(a_i(\beta-\alpha)^i)>v(a_m(\beta-\alpha)^m)$ by maximality of $m$, the quantity in (\ref{eqn22}) is greater than $v(m!)$.
Thus we conclude that 
$$v\left(P^{(m)}(1)\right)=v(m!).$$

Now, from 
$$v(P(1))>2^m\left(v(m!)+\delta\right)=2^m\left(v(P^{(m)}(1))+\delta\right)$$
we must have for some $n< m$
$$v\left(P^{(n)}(1)\right)>2\left(v(P^{(n+1)}(1))+\delta\right).$$

Proposition \ref{prop7} now gives $u\in\cO$ with $P^{(n)}(u)=0$ and $\rvd(u)=\rvd(1)$.  Set $\lambda:=(\beta-\alpha)u+\alpha$.  Since $\beta-\alpha\neq 0$, it follows from
$$P^{(n)}(u)=\frac{(\beta-\alpha)^n}{\sigma}f^{(n)}(\lambda)=0$$
that $f^{(n)}(\lambda)=0$.  Finally, $\rvd(\lambda-\alpha)=\rvd(u(\beta-\alpha))=\rvd(\beta-\alpha)$, as required.
\end{proof}

\subsection{The decomposition}

Like the $m$ in the proof of Proposition \ref{prop21}, 
we will frequently need to refer to the largest degree term carrying the smallest valuation.  
Therefore define
\begin{equation}
\label{dfnm}
m(f,\alpha,S):=\max\left\{i\leq d\mid \exists x\in S\ \forall j\leq d\ \left(v\left(a_i(x-\alpha)^i\right)\leq v\left(a_j(x-\alpha)^j\right)\right)\right\}
\end{equation}
where as before the $a_i$ are the coefficients of the expansion of $f(x)$ around $\alpha$, $f(x)=\sum a_i(x-\alpha)^i$.
Thus, $m(f,\alpha,S)$ is the highest power term in $f$ centered at $\alpha$ which can have minimal valuation (among the other terms of $f$) on $S$.

\begin{prop}
\label{prop22}
If $f(x)\in K[x]$, $\beta\in S$, $v(\beta-\alpha)=\delta$, and $T\subseteq B_{\geq\delta}(\beta)$, then $m(f,\beta,T)\leq m(f,\alpha,S)$.
\end{prop}

\begin{proof}
Let 
$f(x)=\sum\limits_{i=0}^da_i(x-\alpha)^i=\sum\limits_{i=0}^db_i(x-\beta)^i$
and 
$$n:=\max\set{i\leq d}{\all{j\leq d}{v(a_n\left(\beta-\alpha)^n\right)\leq v\left(a_j(\beta-\alpha)^j\right)}}$$ 
(so $n\leq m(f,\alpha,S)$).  Define also $\sigma:=a_n(\beta-\alpha)^n$.

Like in Proposition \ref{prop21}, from $f(x)$ we define the polynomials
$$P_\alpha(x):=\frac{f((\beta-\alpha)x+\alpha)}{\sigma}=\frac{1}{\sigma}\sum\limits_{i=0}^da_i(\beta-\alpha)^ix^i$$
$$P_\beta(x):=\frac{f((\beta-\alpha)x+\beta)}{\sigma}=\frac{1}{\sigma}\sum\limits_{i=0}^db_i(\beta-\alpha)^ix^i$$
so that $P_\alpha\in\cO[x]$ and $\deg\left(\res\left(P_\alpha\right)\right)=n$.  

Furthermore, since $P_\alpha(x+1)=P_\beta(x)$, $P_\beta(x)\in\cO[x]$ and $\deg\left(\res\left(P_\beta\right)\right)=\deg\left(\res\left(P_\alpha\right)\right)=n$ as well.
This implies that
\begin{equation}
\label{eqn221}
0=v(b_n)+n\delta=v\left(b_n(\beta-\alpha)^n\right)<v\left(b_i(\beta-\alpha)^i\right)=v(b_i)+i\delta
\end{equation}
for $i>n$.

Now, taking any $\zeta\in T$, $v(\zeta-\beta)\geq\delta$ combined with (\ref{eqn221}) gives
$$v\left(b_n(\zeta-\beta)^n\right)<v\left(b_i(\zeta-\beta)^i\right)$$
for all $i>n$.  Therefore we have $m(f,\beta,T)\leq n \leq m(f,\alpha,S)$.
\end{proof}

The partition of $K$ is made up of \emph{swiss cheeses}.  Recall that a swiss cheese is a set of the form $B\setminus\left(C_1\cup\ldots\cup C_n\right)$, where $B$ and each $C_i$ are (open or closed) balls, including $K$ itself as well as singletons.  A key property is that the intersection of two swiss cheeses is again a swiss cheese.

\begin{prop}
\label{prop23}
Let $f(x)\in K[x]$ and $S$ be a swiss cheese in $K$.  Then there exist (disjoint) sub-swiss cheeses $T_1,\ldots,T_k\subseteq S$ and $\alpha_1,\ldots,\alpha_k\in K$ such that 
$$S=\bigcup\limits_{i=1}^kT_i$$ 
and for all $x\in T_i$, 
$$v\left(a_{im_i}(x-\alpha_i)^{m_i}\right)\leq v(f(x))\leq v\left(a_{im_i}(x-\alpha_i)^{m_i}\right)+2^{m_i}v(m_i!)$$
where $f(x)=\sum\limits_{n=0}^d a_{in}(x-\alpha_i)^n$ and $m_i=m(f,\alpha_i,T_i)$.

Furthermore the $\alpha_i$ can be chosen from among the roots of the derivatives of $f(x)$.
\end{prop}

\begin{proof}
To begin, choose any root $\alpha$ of a derivative of $f$, and let $f(x)=\sum_{n=0}^da_i(x-\alpha)^i$.  For simplicity, assume that $S$ is a ball $\Bga$.  No generality is lost as a decomposition for $\Bga\supseteq S$ may simply be intersected with $S$ to get the desired result.  In particular, $\alpha\in S$.

The proof proceeds by a double induction, first on $m(f,\alpha,S)$ and then on the number of roots of derivatives of $f$ contained in $S$.  Clearly, if $m(f,\alpha,S)=0$, then $v(f(x))=v(a_0)$ for all $x\in S$.

Now suppose $m(f,\alpha,S)=m$.  Let 
$$D:=\set{\delta\geq\gamma}{\all{i\leq m}{v(a_m)+m\delta \leq v(a_i)+i\delta}}.$$
In other words, $m(f,\alpha,S)=m$ when $v(a_m(x-\alpha)^m)$ is minimal \emph{somewhere} in $S$ (within the set $\set{v(a_i(x-\alpha)^i)}{0\leq i\leq d}$), while $D$ gives those values where it actually \emph{is} minimal.  Define also
$$B_D:=\set{x\in S}{v(x-\alpha)\in D}.$$

$D$ is an initial segment of $[\gamma,\infty)$.  
Indeed, if $\gamma\leq \ep<\delta\in D$ and $i<m$, then
$v(a_i)+i\delta \geq v(a_m)+m\delta$ implies
\begin{equation}
\label{eqn23}
v(a_i)+i\ep>v(a_m)+m\ep,
\end{equation}
so $\ep\in D$ as well.
We need not consider $i>m$, by the maximality of $m$.

In particular, the inequality in (\ref{eqn23}) becomes strict for $\ep<\delta$.
Therefore we have also shown that if $\ep\in D$ is not a maximal element of $D$, 
then for all $x$ such that $v(x-\alpha)=\ep$, 
$$v(f(x))=v(a_m(x-\alpha)^m).$$

This already suffices to prove the claim if $D=[\gamma,\infty)$, so we assume that $D$ is in fact a proper initial segment.
In this case, there is some $i<m$ such that $v(a_m(x-\alpha)^m)>v(a_i(x-\alpha)^i)$ whenever $v(x-\alpha)>\delta$ for every $\delta\in D$.  
Set $\eta :=v(a_i)-v(a_m)$ and note that 
$$S\setminus B_D=\Bga\setminus B_D=B_{>\eta/(m-i)}(\alpha)$$
since for $x\in S$,
\begin{align*} 
x\notin B_D &\Leftrightarrow v(a_i(x-\alpha)^i)<v(a_m(x-\alpha)^m)\\ 
 &\Leftrightarrow v(a_i)-v(a_m)<(m-i)v(x-\alpha).
\end{align*}

Therefore, so far we have:
\begin{enumerate}[(i)] 
\item
if $x\in B_D=\Bga\setminus B_{\geq \eta/(m-i)}(\alpha)$, but $v(x-\alpha)$ is not maximal in $D$, then $v(f(x))=v(a_m(x-\alpha)^m)$ by (\ref{eqn23});
\item
if $x\in S\setminus B_D=B_{>\eta/(m-i)}(\alpha)$, then $m(f,\alpha,B_{>\eta/(m-i)}(\alpha))<m$ and the induction hypothesis applies.
\end{enumerate}
Note that the existence of a maximal element of $D$ depends on the divisibility of $\eta$ by $m-i$ in $V$, but regardless, as mentioned in Section \ref{defandnot}, open (or closed) balls of radius $\eta/(m-i)$ are still definable.

Setting $\delta:=\eta/(m-i)$, now $S$ is the disjoint union of the three swiss cheeses
$$S\ =\ \Bga\setminus B_{\geq\delta}(\alpha)\ \cup\  B_{\geq\delta}(\alpha)\setminus B_{>\delta}(\alpha)\ \cup\ B_{>\delta}(\alpha),$$
the second being empty if $\delta\notin V$.  
On the first of these, as observed above, 
$v(f(x))=v(a_m(x-\alpha)^m)$, 
and on the last, $m(f,\alpha,B_{>\delta}(\alpha))<m$.  
It therefore remains only to consider 
$A:=B_{\geq\delta}(\alpha)\setminus B_{>\delta}(\alpha)$, 
i.e. where $D$ contains a maximal element $\delta$ and $v(x-\alpha)=\delta$.

Let $C$ be the set $\set{x\in A}{v(f(x))>v\left(a_m(x-\alpha)^m\right)+2^mv(m!)}$.  Now the condition on $v(f(x))$ of the proposition also holds on $A\setminus C$, so in fact it only remains to consider $v(f(x))$ on $C$.  Define an equivalence relation $\sim$ on $C$ by $x\sim y \Leftrightarrow v(x-y)>\delta \Leftrightarrow \rv(x-\alpha)=\rv(y-\alpha)$.

Proposition \ref{prop21} shows that each $\sim$-equivalence class in $C$ contains a root $\lambda$ of a derivative of $f$.  Thus, each such equivalence class is of the form $B_{>\delta}(\lambda)$, and in particular, there are finitely many of them.  So $A\setminus C$ is a swiss cheese, and we finally must only prove the claim for a ball $B=B_{>\delta}(\lambda)$.

By Proposition \ref{prop22}, $m(f,\lambda,B)\leq m$.
If in fact $m(f,\lambda,B) < m$, then the induction hypotheses takes effect, and we're done.  However, equality may occur.  In this case, however, note that $\alpha\notin B$ (since $f$ cannot have a collision at $\alpha$ around $\alpha$).  As $\alpha$ was chosen to be a root of a derivative of $f$ and $\alpha\in S$, $B$ contains strictly fewer roots of derivatives of $f$ than $S$.  Thus, in this case the secondary induction hypothesis applies to complete the proof.
\end{proof}

In residue characteristic $0$ the statement of Proposition \ref{prop23} simplifies considerably: since $v(n)=0$ for all integers $n$, on each $T_i$ we get in fact $v(f(x))=v\left(a_{im_i}(x-\alpha_i)^{m_i}\right)$.

Finally, we return to the leading term structures to find that the above decomposition also enables the analysis of $\rvd(f(x))$.  Thanks to Propositions \ref{prop3.5} and \ref{prop4}, this is an immediate consequence of the above proposition.

\begin{prop}
\label{prop25}
Let $f(x)\in K[x]$ be a polynomial of degree $d$ and $0\leq \delta\in V$.  Then there are 
\begin{enumerate}[(i)]
\item
disjoint swiss cheeses $U_1,\ldots,U_k$ partitioning $K=\bigcup\limits_{i=1}^kU_i$,
\item
elements $\alpha_1,\ldots,\alpha_k\in K$, 
\item
and positive integers $q_1,\ldots,q_k\leq (d!)^{2^d}$
\end{enumerate} 
such that for each $i$, if $f(x)=\sum\limits_{j=0}^da_{ij}(x-\alpha_i)^j$ then for all $x\in U_i$, 
$$\rvd\left(P(x)\right)=\rvd\left(\sum\limits_{j=0}^d\rv_{\delta+v(q_i)}(a_{ij})\rv_{\delta+v(q_i)}(x-\alpha_i)^j\right)$$ 
is well-defined.

The $\alpha_1,\ldots,\alpha_k$ can be chosen from among the roots of derivatives of $f$.
\qed
\end{prop}

Though each of the preceding propositions is stated for a single polynomial $f(x)$, the same results will hold for any finite number of polynomials $f_1,\ldots,f_n$.  To obtain the desired decomposition, simply apply the proposition to each $f_i$ separately, and then intersect the resulting partitions to get one which works for all $f_i$ simultaneously.  We are again using the fact that the intersection of finitely many swiss cheeses is a swiss cheese.

\section{Quantifier elimination}
\label{qe}

The methods used in the decomposition of the previous section are reminiscent of those employed by Cohen \cite{coh1} in his decision procedure for the $p$-adics (as well as those of Cluckers and Loeser \cite{cl1} in the context of $b$-minimality).  In fact, these results and techniques can be used to give an effective quantifier elimination, and therefore a decision procedure, for the field relative to the leading term structures.

Unlike in $\bbQ_p$, there can be no quantifier elimination or decision procedure for general henselian valued fields, due to the lack of control over the residue field or value group in the general case.  One could propose that $\bbQ_p$ is decidable precisely because its residue field (a finite field) and value group (Presburger arithmetic) are.  

The objective, then, turns to relative results.  As noted in the Introduction, Kuhlmann \cite{kuh1} proved that in the leading term language, the theory of a henselian valued field of characteristic $0$ eliminates quantifiers over the field sort.  In this section, we give a new proof of Kuhlmann's theorem which yields not only the relative quantifier elimination, but an explicit procedure for eliminating field-sorted quantifiers.  

This implies a relative decision procedure in the sense that if the leading term structures are themselves decidable, then the valued field as a whole is decidable; or alternatively, if we allow ourselves access to an oracle for the leading term structures, then we can construct a decision procedure for the valued field.

Let us point out also that the quantifier elimination fails relative to the residue field and value group.  The leading term language is a necessity here.  To see this, consider the elements $x_1=t^2$ and $x_2=2t^2$ in the field $\bbQ((t))$.  Although $x_1$ is a square while $x_2$ is not, since both are transcendental over $\bbQ$ and both have identical residue and valuation, $x_1$ and $x_2$ satisfy precisely the same field-quantifier-free formulas in the standard three-sorted language.  

One could circumvent this by adding other additional structure such as a cross-section of the value group or an angular component map (see for example \cite{yin1}) on the field.  However such a language would be strictly stronger than the leading term language in that it could interpret the leading term structures, but also contains a definable subset isomorphic to the value group (namely, the value group sort itself).

The first step in the quantifier elimination comes from deciding questions about when certain finite sets of balls have a non-empty intersection.

\begin{prop}
\label{prop51}
Let $z_i, a_i\in K$, $0\leq \delta_i \in V$ for $i\leq n$.  
The formula
\begin{equation*}
\ex{x}{\bigwedge\limits_{i\leq n} \rv_{\delta_i}(z_i)=\rv_{\delta_i}(x-a_i)}
\end{equation*}
is equivalent to a formula with no field-sorted quantifiers over the parameters
$\rv_{\delta_i}(z_i)$, $\rv_{\delta_i}(a_i-a_j)$, and $\delta_i$ (or, more precisely, an element of value $\delta_i$).
\end{prop}

\begin{proof}
Notice that the set of $x$ satisfying $\rv_{\delta_i}(z_i)=\rv_{\delta_i}(x-a_i)$ is in fact equal to the open ball $B_i:=B_{>v(z_i)+\delta_i}(z_i+a_i)$.
So what is sought is a means of testing for nonemptiness of the intersection of the balls $B_i$.  Since finitely many balls having pairwise nonempty intersections implies a nonempty intersection, it will be sufficient to do so for the intersection of two balls.  Thus we may assume $n=2$.

Let us assume also that $v(z_1)+\delta_1 \leq v(z_2)+\delta_2$.  
This implies that $B_1 \cap  B_2 \neq \emptyset$ iff $B_1 \supseteq B_2$ iff $z_2+a_2 \in B_1$ iff $v(z_1+a_1-z_2-a_2)>v(z_1)+\delta_1$.

\begin{itemize}
\item
\emph{Case 1:} $v(z_1)\leq v(a_1-a_2)$, $v(z_1) \leq v(z_2)$, and $\delta_1 \leq \delta_2$.  

Then, by Proposition \ref{prop4.5}, $v(z_1+a_1-z_2-a_2)>v(z_1)+\delta_1$ is equivalent to 
\begin{equation*}
\ex{\mathbf{w}_1,\mathbf{w}_2\in \RV_{\delta_1}}
{
\begin{array}{c}
v(\mathbf{w}_1)\neq v(\mathbf{w}_2)\ \wedge \\   
\rv_{\delta_1}(z_1)-\rv_{\delta_1}(z_2)+\rv_{\delta_1}(a_1-a_2)\approx\bfw_1\ \wedge \\
\rv_{\delta_1}(z_1)-\rv_{\delta_1}(z_2)+\rv_{\delta_1}(a_1-a_2)\approx\bfw_2
\end{array}
}
\end{equation*}
since the sum in $\RV_{\delta_1}$ at least determines the valuation except when 
$$v(z_1-z_2+a_1-a_2)>\min\left\{v(z_1),v(z_2),v(a_1-a_2)\right\}+\delta_1=v(z_1)+\delta_1.$$

\item
\emph{Case 2:} $v(z_1)\leq v(a_1-a_2)$, $v(z_1) \leq v(z_2)$, and $\delta_1 > \delta_2$.  

This time, although $\rv_{\delta_1}(z_2)$ is no longer uniquely determined from 
$\rv_{\delta_2}(z_2)$, 
$v(z_1+a_1-z_2-a_2)>v(z_1)+\delta_1$ is equivalent to 
$$
\forall \mathbf{u}\in\RV_{\delta_1}\exists \mathbf{w}_1,\mathbf{w}_2\in \RV_{\delta_1} \\
\hspace{3 in}
$$
$$
\hspace{.4 in}
\left(
\rv_{\delta_2}(\mathbf{u})=\rv_{\delta_2}(z_2) \rightarrow
\left(
\begin{array}{c}
v(\mathbf{w}_1)\neq v(\mathbf{w}_2)\ \wedge \\   
\rv_{\delta_1}(z_1)-\mathbf{u}+\rv_{\delta_1}(a_1-a_2)\approx\bfw_1\ \wedge \\
\rv_{\delta_1}(z_1)-\mathbf{u}+\rv_{\delta_1}(a_1-a_2)\approx\bfw_2
\end{array}
\right)
\right)
$$
because $\rv_{\delta_2}(\mathbf{u})=\rv_{\delta_2}(z_2)$ implies that $\rv_{\delta_1}(z_1)-\mathbf{u}=\rv_{\delta_1}(z_1)-\rv_{\delta_1}(z_2)$.  

To see this, let $\mathbf{u}=\rv_{\delta_1}(u)$ and note that the inequality $v(z_1)<v(z_2)=v(u)$ must in fact be strict.  Thus $\rv_{\delta_1}(z_1)-\mathbf{u}=\rv_{\delta_1}(z_1-u)$ and $\rv_{\delta_1}(z_1)-\rv_{\delta_1}(z_2)=\rv_{\delta_1}(z_1-z_2)$ as well-defined sums.  
Now
$$v((z_1-u)-(z_1-z_2))=v(z_2-u)>v(z_2)+\delta_2\geq v(z_1)+\delta_1$$
by $\rv_{\delta_2}(u)=\rv_{\delta_2}(z_2)$.

Now argue as in Case 1.

\item
\emph{Case 3:} $v(z_1)\leq v(a_1-a_2)$ and $v(z_2) < v(z_1)$.

This implies $v(z_1+a_1-z_2-a_2)=v(z_2)<v(z_1)+\delta_1$, so this case is trivial.

\item
\emph{Case 4:} $v(a_1-a_2) < v(z_1)$.

In this case, $\rv_{\delta_1}(z_1)+\rv_{\delta_1}(a_1-a_2)$ is well-defined.
Then 
$$\ex{x}{\rv_{\delta_1}(z_1)=\rv_{\delta_1}(x-a_1)\wedge \rv_{\delta_2}(z_2)=\rv_{\delta_2}(x-a_2)}$$
holds if and only if 
$$\ex{x}{\rv_{\delta_1}(z_1)+\rv_{\delta_1}(a_1-a_2)=\rv_{\delta_1}(x-a_2)\wedge \rv_{\delta_2}(z_2)=\rv_{\delta_2}(x-a_2)}.$$
If $\delta_1 \leq \delta_2$ this is equivalent to
$$\rv_{\delta_1}(z_1)+\rv_{\delta_1}(a_1-a_2)=
\rv_{\delta_1}\left(\rv_{\delta_2}(z_2)\right)$$
(witnessed when the above holds by $x=z_2+a_2$),
while if $\delta_2 < \delta_1$ it is equivalent to
$$\rv_{\delta_2}\left( \rv_{\delta_1}(z_1)+\rv_{\delta_1}(a_1-a_2) \right)
=\rv_{\delta_2}(z_2)$$
(witnessed by $x=z_1+a_1$).
\end{itemize}

The desired formula will then be the disjunction over all these cases.
\end{proof}

In fact, we will need the above result to apply more generally to formulas involving the leading terms of polynomials linear in $x$.

\begin{prop}
\label{prop52}
Let $z_i, a_i, b_i\in K$ with $a_i\neq 0$.  
The formula
\begin{equation}
\label{eqn52}
\ex{x}{\bigwedge\limits_{i\leq n} \rv_{\delta_i}(z_i)=\rv_{\delta_i}(a_ix-b_i)}
\end{equation}
is equivalent to a formula with no field-sorted quantifiers over parameters
$\rv_{\delta_i}(z_i)$, $\rv_{\delta_i}(a_j)$, $\rv_{\delta_i}(a_ib_j-a_jb_i)$, and $\delta_i$.
\end{prop}

\begin{proof}
This is easily adapted from \ref{prop51} by applying the proposition after factoring out $\rv_{\delta_i}(a_i)$ in (\ref{eqn52}).
\end{proof}

Proposition \ref{prop52} forms the basis for an induction on the maximum degree of a polynomial appearing as a leading term.  The relative quantifier elimination essentially uses the linearization of the leading terms of polynomials to push questions about the existence of field elements into the leading term structures.  

One consequence of this approach is that we need not make any assumptions on the formula in the $\RV$ structures.  Indeed, we may allow any additional structure (such as a cross section, or an expansion to $\RV^{\text{eq}}$) on the leading terms.  The important point is that the field sort carries only the usual ring language and the map(s) $\rv_\delta$.

The basic situation, therefore, would be a two-sorted structure $(K, \RV)$ in residue characteristic $0$, and a many-sorted structure $(K, \RV_0, \RV_{v(p)}, \RV_{v(p^2)},\ldots)$ when $\chr(R)=p>0$.  In full generality, however, the language can include any expansion on the leading term sorts of these basic languages.

\begin{prop}
\label{prop53}
Let $T$ be the theory of a characteristic $0$ henselian field in a language of the kind described above.  Then $T$ eliminates field-sorted quantifiers.
\end{prop}

\begin{proof}
We break the proof up into several steps, each of which further reduces the class of formulas needing to be considered.  To mitigate a logjam of indices, the notation is reset at each step, so that $f(x)$ in Step 2 is not necessarily the same as $f(x)$ in Step 1, but only the syntax of the formula under consideration is maintained.

We begin with an existential formula of the form
\begin{equation}
\label{eqn33}
\exists x\in K\big(\varphi\left(\rv_{\delta_1}(f_1(x,\bar{u})),\ldots,\rv_{\delta_n}(f_n(x,\bar{u}))\right)\big)
\end{equation}
where $\varphi$ is some predicate definable (with $\RV$-sorted parameters, possibly in an expanded language) in 
$\RV_{\delta_1}\times\ldots\times\RV_{\delta_n}$, the $f_i$ are polynomials over $K$,
and all field-sorted free variables are among $\bar{u}$ (which we henceforth suppress from the notation).
It suffices to show that this is equivalent to a field-quantifier-free formula.

We proceed by induction on $m:=\max_{i\leq n}\left\{\deg(f_i(x))\right\}$.  If $m=0$, the result is trivial.

If $m=1$, i.e. each $f_i$ is linear in $x$, rewrite (\ref{eqn33}) as
\begin{equation}
\label{eqn335}
\ex{\mathbf{z_i}\in\RV_{\delta_i}}
{
\varphi\left(\mathbf{z_1},\ldots,\mathbf{z_n}\right)\ \wedge\ 
\ex{x\in K}{\bigwedge\limits_{i\leq n} \mathbf{z_i}=\rv_{\delta_i}(f_i(x))}
}.
\end{equation}
Now Proposition \ref{prop52} applies to eliminate the quantifier $\exists x$.

{\bf Step 1:} \emph{From (\ref{eqn33}) to formulas of the form
\begin{equation}
\label{eqn531}
\exists x\in K\big(f(x)=0 \wedge \varphi\left(\rv_{\delta_1}(g_1(x)),\ldots,\rv_{\delta_n}(g_n(x))\right)\big)
\end{equation}
with $\deg(f(x))\leq m$.
}

Proposition \ref{prop25} gives a partition 
$K=\bigcup\limits_{j=1}^mS_j$ and for each $i\leq n, j\leq m$:
\begin{itemize}
\item
a root $\alpha_{ij}$ 
of some derivative $g_{ij}(x)$
of $f_i$ (including possibly $f_i$ itself), 
\item
and positive integers $q_{ij}\leq 2^mv(m!)$ 
\end{itemize}
such that for all $x\in S_j$ and $i\leq n$,
$\rv_{\delta_i}(f_i(x))$ can be computed as the well-defined image in $\RV_{\delta_i}$ of a polynomial function of $\rv_{\delta_i+v(q_{ij})}(x-\alpha_{ij})$.

The roots $\alpha_{1j},\ldots,\alpha_{nj}$ also serve as centers of the balls comprising the swiss cheeses $S_j$.

In this way, the formula in (\ref{eqn33}) is equivalent to one of the form
\begin{equation}
\label{eqn34}
\begin{array}{c}
\exists y_{11},\ldots,y_{nm}\in K
\ 
\Bigg(\bigwedge\limits_{i,j} g_{ij}(y_{ij})=0\ \wedge \\
\ex{x\in K}{\bigvee\limits_j
\varphi_j\left(\rv_{\delta_1+v(q_{1j})}(x-y_{1j}),\ldots,\rv_{\delta_n+v(q_{nj})}(x-y_{nj})\right)}\Bigg).
\end{array}
\end{equation}

Specifically, $\varphi_j$ will express that
$x\in S_j$, that
$$\bfu_{ij}:=\sum\limits_{k=0}^{d_i}\rv_{\delta_i+v(q_{ij})}(a_{ijk})\rv_{\delta_i+v(q_{ij})}(x-y_{ij})^k$$
(given $f_i=\sum\limits_{k=0}^{d_i}a_{ijk}(x-y_{ij})^k$, so $a_{ijk}$ is a function of $y_{ij}$) is well-defined for each $i$, and that $\varphi$ holds with 
$\bfu_{ij}$
substituted for each $f_i$.

In (\ref{eqn34}) the bound variable $x$ occurs only linearly, so it can be eliminated as shown above.  This produces an equivalent formula in the form
\begin{equation}
\label{eqn344}
\ex{\left(y_{ij}\right)_{i,j}\in K}{\left(\bigwedge\limits_{i,j} g_{ij}(y_{ij})=0\right)
\wedge 
\psi\left(
\rv_{\gamma_1}(h_1(\bar{y})),\ldots,\rv_{\gamma_\ell}(h_\ell(\bar{y}))
\right)
}
\end{equation}
with $h_1,\ldots,h_\ell$ being polynomials and $\psi$ an $\RV$ formula.
So it remains to show that the quantifiers $\exists y_{11},\ldots,y_{mn}$ can be eliminated in such a formula.

In fact we may do so one quantifier at a time, so it will suffice to consider a formula of the form
\begin{equation}
\label{eqn35}
\exists y\in K\big(g(y)=0 \wedge \psi\left(\rv_{\gamma_1}(h_1(y)),\ldots,\rv_{\gamma_\ell}(h_\ell(y))\right)\big)
\end{equation}
with $\deg(g(y))\leq m$.  This completes Step 1.

{\bf Step 2:} \emph{From (\ref{eqn531}) to formulas of the form
\begin{equation}
\label{eqn532}
\exists x\in K\big(f(x)=0 \wedge \varphi\left(\rv_{\delta_1}(x-z_1),\ldots,\rv_{\delta_n}(x-z_n)\right)\big)
\end{equation}
with $\deg(f(x))\leq m$, and the $z_i$ free variables.
}

First of all, in (\ref{eqn531}),
$$\exists x\in K\big(f(x)=0 \wedge \varphi\left(\rv_{\delta_1}(g_1(x)),\ldots,\rv_{\delta_n}(g_n(x))\right)\big),$$
each $g_i(x)$ can be replaced with its remainder on division by $f(x)$
(in applying the euclidean algorithm, it will be necessary to multiply through by powers of the leading coefficient of $f$).  
Thus it may be assumed that 
$$\deg(g_i(x))<\deg(f(x))\leq m$$
for each $i\leq m$.  If the latter inequality were strict, of course, the induction hypothesis would finish the proof.

Otherwise, if we have equality, let us apply the decomposition of Proposition \ref{prop25} a second time relative to $g_1(x),\ldots,g_n(x)$.  The result is another formula equivalent to (\ref{eqn531}) taking the form
\begin{equation}
\label{eqn36}
\ex{\left(z_i\right)_{i\leq k}\in K}{\bigwedge\limits_{i\leq k} h_i(z_i)=0 \wedge
\exists x\big(f(x)=0 \wedge \psi\left(
\rv_{\gamma_1}(x-z_1),\ldots,\rv_{\gamma_k}(x-z_k)
\right)
\big)
}
\end{equation}
with 
$\deg(h_i(z_i))\leq \max\limits_{j\leq n}\left\{\deg(g_j(x))\right\} < m$ for every $i$.  Now, it will suffice to eliminate the quantifier $\exists x$ from the subformula
\begin{equation}
\label{eqn37}
\exists x\big(f(x)=0 \wedge \psi\left(
\rv_{\gamma_1}(x-z_1),\ldots,\rv_{\gamma_k}(x-z_k)
\right)
\big)
\end{equation}
since then we would be in the situation of (\ref{eqn344}) except now with the degrees of the $h_i(z_i)$ \emph{strictly} less than $m$.

This completes Step 2.

{\bf Step 3:} \emph{From (\ref{eqn532}) to formulas of the form
$$
\ex{x\in K}{f(x)=0 \wedge \varphi\left(\rv_{\delta}(x-z)\right)}
$$
with $z$ a free variable, $\deg(f(x))=d\leq m$, and $f(x)$ coprime to $f^{(i)}(x)$ for $1\leq i\leq d$.
}

In
$$\exists x\in K\big(f(x)=0 \wedge \varphi\left(\rv_{\delta_1}(x-z_1),\ldots,\rv_{\delta_n}(x-z_n)\right)\big),$$
suppose that we had $v(x-z_i)\geq v(x-z_k)$ and $\delta_j\geq\delta_k$ for all $k\leq n$.  Then for each $k$, $\rv_{\delta_k}(x-z_k)=\rv_{\delta_j\rightarrow\delta_k}\left(\rv_{\delta_j}(x-z_i)+\rv_{\delta_j}(z_i-z_k)\right)$ is well-defined.  (Since 
$v(z_i-z_k)\geq v(x-z_k)$, $$v(x-z_k)\geq \min\left\{v(x-z_i),v(z_i-z_k)\right\}\geq v(x-z_k)$$ implies equality and hence well-definition of $\rv_{\delta_j}(x-z_k)=\rv_{\delta_j}(x-z_i)+\rv_{\delta_j}(z_i-z_k)$.)

Thus $\varphi\left(
\rv_{\delta_1}(x-z_1),\ldots,\rv_{\delta_n}(x-z_n)\right)$ depends only on $\rv_{\delta_j}(x-z_i)$ and the parameters $\rv_{\delta_j}(z_i-z_k)$, and in this case we may write (\ref{eqn532}) as
\begin{equation}
\label{eqn38}
\ex{x}{f(x)=0 \wedge \psi\left(\rv_{\delta_j}(x-z_i)\right)}.
\end{equation}
If we can eliminate the $\exists x$ in this formula, then by taking the disjunction over the possible cases of which $v(x-z_i)$ is largest, we will be done.

Regarding the coprimality condition, if for some $i\geq 1$ we have $\gcd(f(x),f^{(i)}(x))=g(x)$ and $f(x)=g(x)h(x)$, (\ref{eqn38}) is equivalent to
$$\ex{x}{g(x)=0 \wedge \psi\left(\rv_{\delta_j}(x-z_i)\right)}\vee
\ex{x}{h(x)=0 \wedge \psi\left(\rv_{\delta_j}(x-z_i)\right)},$$
and $\deg(g(x)),\deg(h(x))<m$.

Therefore we may also assume $\gcd(f(x),f^{(i)}(x))=1$ for all $1\leq i\leq d$, finishing Step 3.

{\bf Step 4:} \emph{
Eliminating the quantifier $\exists x$ from the formula
\begin{equation}
\label{eqn533}
\exists x\in K\big(f(x)=0 \wedge \varphi\left(\rv_{\delta}(x-z)\right)\big)
\end{equation}
when $\deg(f(x))=d\leq m$ and $\gcd(f(x),f^{(i)}(x))$ for $1\leq i\leq d$.
}

Suppose $f(x)=\sum\limits_{i=0}^da_i(x-z)^i$.  If $a_0=0$, then $z$ is a root of $f$ and we may check whether $\varphi$ holds on $\rv_\delta(z-z)=\infty$.

Let $\gamma=2^d(v(d!)+\delta)$ and $\chi(\bfy)$ be the formula
$$\ex{\bfu_1,\bfu_2}{\sum\limits_{i=0}^d\rvg(a_i)\bfy^i\approx\bfu_1\wedge
\sum\limits_{i=0}^d\rvg(a_i)\bfy^i\approx\bfu_2 \wedge v(\bfu_1)\neq v(\bfu_2)}.$$
So $f$ has a collision of severity $>\gamma$ at $x$ around $z$ if and only if $\chi(\rvg(x-z))$ (by Proposition \ref{prop4.5}).

Let also $\varrho(y_1,\ldots,y_n)$ be a field-quantifier-free formula equivalent to
$$\left(\bigwedge\limits_{i=1}^n\bigvee\limits_{j=1}^d f^{(j)}(y_i)=0\right)\wedge \left(\bigwedge\limits_{j=1}^d\left(\neg\ex{y\in K}{f^{(j)}(y)=0\wedge \bigwedge_i y\neq y_i}\right)\right)$$
stating that $y_1,\ldots,y_n$ are all the roots of the proper derivatives $f',f'',\ldots,f^{(d)}$.  Such a formula must exist by the induction hypothesis.  We can take $n$ to be as large as necessary, no more than $\frac{d(d-1)}{2}$.

Now, consider the formula
\begin{equation}
\label{eqn39}
\begin{array}{c}
\exists y_1,\ldots,y_n\exists x \Big(\varrho(y_1,\ldots,y_n)\wedge\big(\bigwedge_{i=1}^n\chi(\rvg(x-y_i)\big)
\wedge\varphi(\rvd(x-z))\Big).
 \end{array}
\end{equation}

We claim that (\ref{eqn39}) is equivalent to (\ref{eqn533}).  In fact, if $f$ has a root at $x$ and $\lambda$ is a root of $f^{(i)}$ ($1\leq i<d$), then $f(\lambda)\neq 0$ by coprimality,
and so the constant term of $f$ recentered around $\lambda$ is nonzero.  This implies that $f$ still has a collision at $x$ around $\lambda$ (of infinite severity).

Suppose conversely that (\ref{eqn39}) holds.  Proposition \ref{prop22} (and the proof of Proposition \ref{prop23}) implies that there is a $\lambda$, which is a root of one of $f,f',f'',\ldots,f^{(d)}$, for which $\rvd(x-z)=\rvd(\lambda-z)$ and $f$ does not have a collision at $x$ around $\lambda$.  Since $\chi$ holds on each $\rvg(x-y_i)$, therefore, this $\lambda$ must be a root of $f$ itself.  In other words, (\ref{eqn39}) implies that there is a root $\lambda$ of $f$ for which $\rvd(x-z)=\rvd(\lambda-z)$, and $\varphi$ holds for this leading term $\rvd(\lambda-z)$.  This shows that (\ref{eqn533}) and (\ref{eqn39}) are equivalent.

In (\ref{eqn39}), the quantifier $\exists x$ can be eliminated as in (\ref{eqn335}), since $x$ appears only linearly.  Likewise, each quantifier $\exists y_i$ can also be eliminated by the induction hypothesis, because $\deg(f^{(i)})<\deg(f)$.

Taking the disjunction over all these cases, we have succeeded in eliminating the field-sorted quantifier in (\ref{eqn533}), and this finishes the proof.
\end{proof}

Since, looking back over the proof of Proposition \ref{prop23}, we have an effective algorithm for producing the swiss cheese decomposition, the above proof gives an effective algorithm for producing a field-quantifier-free formula from any formula in the leading term language.  Assuming formulas in the leading term sorts are decidable, therefore, we may use this to devise a decision procedure for formulas over the valued field, and we have proved

\begin{prop}
The theory of a henselian valued field with $\chr(K)=0$ is decidable relative to an oracle for the leading term structures $\langle\RV_{v(n)}\rangle_{n\in \bbN}$, or equivalently, as long as these structures are decidable.
\qed
\end{prop}

\section{Definable subsets of $K$}
\label{subsets}

In this section, the goal is to use the quantifier elimination and decomposition to give a characterization of definable subsets of $K$.  This provides the promised analogue of the theorem of Holly \cite{hol1} on canonical forms for sets definable (in one variable) in algebraically closed valued fields.

\begin{prop}
\label{prop41}
Suppose $S\subseteq K$ is definable over $A$.  Then there are
$\alpha_1,\ldots,\alpha_k\in \acl(A)$ and a subset 
$D\subseteq \RV_{\delta_1}\times\ldots\times\RV_{\delta_k}$ definable over $\acl(A)$
such that
$$S=\set{x\in K}{\langle \rv_{\delta_1}(x-\alpha_1),\ldots,\rv_{\delta_k}(x-\alpha_k)\rangle\in D}.$$

As before, if $\chr(R)=0$, we may take $\delta_i=0$ for all $i$; if $\chr(R)=p>0$, then the 
$\delta_i$ can be taken among $v(p^n)$ for $n\in\mathbb{N}$.
\end{prop}

\begin{proof}
The elimination of field-sorted quantifiers from Proposition \ref{prop53} implies that $S$ is definable by a formula of the form
\begin{equation}
\label{eqn41}
\varphi\left(\rv_{\delta_1}(f_1(x)),\ldots,\rv_{\delta_k}(f_k(x))\right)
\end{equation}
with $\varphi$ being a formula over the leading term sorts and each $f_i$ a polynomial with coefficients over $A$.

Applying the decomposition of Proposition \ref{prop25}, there are swiss cheeses $U_1,\ldots,U_m$ partitioning $K$, for each $i\leq k$ $\RV$-polynomials $t_{i1},\ldots,t_{im}$ (over $\acl(A)$), and for each $i\leq k$ and
$j\leq m$ field elements $\alpha_{ij}\in\acl(A)$ such that (\ref{eqn41}) is equivalent to
$$
\bigvee\limits_{j=1}^m
\left(
x\in U_j\ \wedge 
\varphi\left(
t_{1j}[\rv_{\delta_{1j}}(x-\alpha_{1j})],\ldots,
t_{kj}[\rv_{\delta_{kj}}(x-\alpha_{kj})]
\right)
\right)
$$
(with each $\delta_{ij}=\delta_i+v(p^n)$, some integer $n$).  For each $i\leq k$ define $\gamma_i:=\max_{j\leq m}\{\delta_{ij}\}$.  Since every $t_{ij}[\rv_{\delta_{ij}}(x-\alpha_{ij})]$ can be computed as $t_{ij}[\rv_{\gamma_{i\rightarrow\delta_{ij}}}(\rv_{\gamma_i}(x-\alpha_{ij}))]$, it may without loss of generality be assumed that $\delta_{ij}=\gamma_i$ for all $i,j$.

The condition $x\in U_j$ is definable in $\RV$ with parameters of the form $\rv(x-\beta)$.  Without loss of generality we take $\beta$ to be among the $\alpha_{ij}$ (so that $x\in U_j$ is an $\RV$-definable condition on $\rv_{\gamma_i}(x-\alpha_{ij})$, some $i,j$), and let 
$\psi_j$
be the formula over the leading term sorts expressing 
$$\psi_j(\mathbf{x}_1,\ldots,\mathbf{x}_k)\Longleftrightarrow 
x\in U_j\ \wedge 
\varphi\left(
t_{1j}[\mathbf{x}_1],\ldots,
t_{kj}[\mathbf{x}_k]
\right).$$
Thus each $\psi_j$ is a formula over $\RV_{\gamma_1}\times\ldots\times\RV_{\gamma_k}$.

Finally, letting $\chi$ be the formula $\bigvee\psi_j$ and $D$ be the set in 
$\RV_{\gamma_1}\times\ldots\times\RV_{\gamma_k}$ defined by $\chi$,
we have
$$S=
\set{x\in K}{\langle \rv_{\gamma_1}(x-\alpha_1),\ldots,\rv_{\gamma_k}(x-\alpha_k)\rangle\in D}$$
as required.
\end{proof}

Holly's swiss cheeses in algebraically closed valued fields arise as boolean combinations of a finite number of balls.  This can be seen as the combination of a pullback of a finite set (from the residue field) and an interval (the value group).  It is a consequence of strong minimality and o-minimality that these are all the sets definable in residue field and value group.  As pointed out in the Introduction, it is unavoidable in the general henselian setting that we must allow for pullbacks of arbitrary definable sets $D$ of the leading term structures, which could be very complicated.

The pullback of an interval in the value group itself will produce a ball (or, more accurately, an annulus) around $0$.  Shifting to balls centered elsewhere in the algebraically closed case can be taken as analogous to our linear shifting by $\langle\alpha_1,\ldots,\alpha_k\rangle$.

To obtain a one-dimensional elimination of imaginaries in \cite{hol2} (`$1$-prototypes'), Holly introduces a new sort for the balls.  It follows by the same reasoning that henselian valued fields of characteristic $0$ admit $1$-prototypes in the leading term language after adding new sorts for definable sets of the form 
$$\set{x\in K}{\langle \rv_{\delta_1}(x-\alpha_1),\ldots,\rv_{\delta_k}(x-\alpha_k)\rangle\in D}.$$

In more dimensions, it is an immediate consequence of quantifier elimination that definable subsets of $K^n$ take the form
\begin{equation}
\label{eqn44}
\set{\langle x_1,\ldots,x_n\rangle\in K}{\langle \rv_{\delta_1}(f_1(\bar{x})),\ldots,\rv_{\delta_k}(f_k(\bar{x}))\rangle\in E}
\end{equation}
where $E$ is definable in $\RV_{\delta_1}\times\ldots\times\RV_{\delta_k}$ and each $f_i\in K[x_1,\ldots,x_n]$.

One could then obtain an essentially trivial elimination of imaginaries by including new sorts consisting of the sets (\ref{eqn44}).  An approach towards a more satisfying solution of the elimination of imaginaries problem may be to give a necessary and sufficient subclass of the polynomials $f_i$.  

For example, one could hope to show that every definable set can be coded in terms of sets of the form (\ref{eqn44}) with the $f_i$ being affine transformations of $K^n$.  This seems overly optimistic, but if true would provide a suitable henselian analogy to Haskell, Hrushovski, and Macpherson's elimination of imaginaries for algebraically closed valued fields \cite{hhm1} in terms of definable modules and torsors over $\cO$.

\end{document}